\documentclass[oneside,a4paper]{amsart}
\usepackage{amsmath,amssymb,amsthm,longtable,mathrsfs,amscd,multirow,latexsym}
\usepackage[dvipdfmx]{graphicx}
\usepackage{longtable}
\usepackage{stmaryrd}
\usepackage{tikz}
\usetikzlibrary{calc}

\usepackage[top=30truemm,bottom=30truemm,left=25truemm,right=25truemm]{geometry}

\newtheorem{dfn}{Definition}[section]
\newtheorem{thm}{Theorem}[section]
\newtheorem{prp}[thm]{Proposition}
\newtheorem{lmm}[thm]{Lemma}

\newtheorem{rem}[thm]{Remark}

\makeatletter
    
    \@addtoreset{equation}{section}
\makeatother
\makeatletter
  \newcommand{\subsubsubsection}{\@startsection{paragraph}{4}{\z@}%
    {1.0\Cvs \@plus.5\Cdp \@minus.2\Cdp}%
    {.1\Cvs \@plus.3\Cdp}%
    {\reset@font\sffamily\normalsize}
  }
  \makeatother
\allowdisplaybreaks[4]

\newcommand{\rs}{\circ}
\newcommand{\pl}{\prec}
\newcommand{\pr}{\succ}

\newcommand{\bsg}{{\boldsymbol{g}}}

\newcommand{\bbB}{\mathbb{B}}
\newcommand{\bbR}{\mathbb{R}}

\newcommand{\mcD}{\mathcal{D}}
\newcommand{\mcM}{\mathcal{M}}
\newcommand{\mcP}{\mathcal{P}}
\newcommand{\mcQ}{\mathcal{Q}}
\newcommand{\mcR}{\mathcal{R}}
\newcommand{\mcS}{\mathcal{S}}

\newcommand{\sfC}{\mathsf{C}}
\newcommand{\sfP}{\mathsf{P}}
\newcommand{\sfR}{\mathsf{R}}

\newcommand{\tif}{\tilde{f}}
\newcommand{\tiomega}{\tilde{\omega}}

\newcommand{\spW}{\mathrm{Alg}(W)}

\newcommand{\words}{\overline{W}}

\newcommand{\id}{\mathrm{id}}

\newcommand{\supp}{\mathop{\text{\rm supp}}}

\title{Iterated paraproducts and iterated commutator estimates in Besov spaces}

\author{Masato Hoshino}
\address{Faculty of Mathematics, Kyushu University}
\email{hoshino@math.kyushu-u.ac.jp}

\date{}

\begin{document}

\begin{abstract}
In the previous study \cite{Hos19}, the author provided an algebraic proof of the multicomponent commutator estimate in Besov spaces $C^\alpha=B_{\infty,\infty}^\alpha$ with $0<\alpha<1$.
In this paper, we extend that result to general Besov spaces $B_{p,q}^\alpha$ with $p,q\in[1,\infty]$ and $0<\alpha<1$.
\end{abstract}

\maketitle

\section{Introduction}

It is well known that, the definition of Besov space in Euclidean space $\bbR^d$ by Littlewood-Paley theory is equivalent to the one based on the estimate of Taylor remainder, when the regularity parameter is positive.
See \cite[Theorem 2.36]{BCD11} for instance.
Especially, Besov space $B_{\infty,\infty}^\alpha(\bbR^d)$ is the same as H\"older space $C^\alpha(\bbR^d)$ if $\alpha$ is a positive noninteger.

In \cite{Hos19}, the author showed a similar equivalence result for \emph{Bony's paraproduct} and its iterated versions.
For any distributions $f,g\in\mcS'(\bbR^d)$, the paraproduct $f\pl g$ is defined via Littlewood-Paley theory, so this is not a local operator. Nevertheless, in the case $f\in C^\alpha(\bbR^d)$ and $g\in C^\beta(\bbR^d)$ with $0<\alpha,\beta$ and $\alpha+\beta<1$, the previous result \cite[Theorem 3.1]{Hos19} implies
\begin{align}\label{intro:remainder of parapro}
(f\pl g)(y)=(f\pl g)(x)+f(x)(g(y)-g(x))+O(|y-x|^{\alpha+\beta}).
\end{align}
Conversely, we can show that the function $h$ of such a local behavior is essentially the same as $f\pl g.$
In \cite{Hos19}, the author studied a generalized version of \eqref{intro:remainder of parapro} for the iterated paraproducts, and as a consequence, provided an algebraic proof of the \emph{commutator estimate} \cite[Lemma~2.4]{GIP15}, which has an important role in the theory of paracontrolled calculus \cite{GIP15}.

In this paper, we consider the Besov type extension of the results in \cite{Hos19}.
First we show the estimate like \eqref{intro:remainder of parapro}, see Theorem \ref{3 thm main} below. The result is no longer a uniform bound on $\bbR^d$, but an $L^pL^q$ type estimate of Taylor remainder.
As a consequence, we also show the commutator estimate in Besov spaces, stated as below.
Commutators discussed in this paper is defined as follows.

\begin{dfn}
For any functions $\xi,f_1,f_2,\dots$ in $\mcS(\bbR^d)$, define
\begin{align*}
\sfC(f_1,\xi)&:=f_1\succeq\xi\ (:=f_1\xi-f_1\pl\xi),\\
{\sf C}(f_1,f_2,\xi)&:={\sf C}(f_1\pl f_2,\xi)-f_1{\sf C}(f_2,\xi),\\
{\sf C}(f_1,\dots,f_n,\xi)&:={\sf C}(f_1\pl f_2,f_3,\dots,f_n,\xi)-f_1{\sf C}(f_2,f_3,\dots,f_n,\xi).
\end{align*}
\end{dfn}

We denote by $B_{p,q}^{\alpha,0}$ the closure of $\mcS(\bbR^d)$ in the space $B_{p,q}^\alpha(\bbR^d)$.
The following theorem is a generalization of \cite[Theorem 4.2]{Hos19} onto Besov norms.

\begin{thm}\label{4 thm commutator estimate}
Let $\alpha_1,\dots,\alpha_n\in(0,1)$ and $\alpha_\circ<0$ be such that
\begin{align*}
&\alpha_1+\cdots+\alpha_n<1,\\
&\alpha_2+\cdots+\alpha_n+\alpha_\circ<0<\alpha_1+\cdots+\alpha_n+\alpha_\circ,
\end{align*}
and let $\alpha:=\alpha_1+\cdots+\alpha_n+\alpha_\circ$.
Let $p_1,\dots,p_n,p_\circ,q_1,\dots,q_n,q_\circ\in[1,\infty]$ be such that
$$
\frac1p:=\frac1{p_1}+\dots+\frac1{p_n}+\frac1{p_\circ}\le1,\quad
\frac1q:=\frac1{q_1}+\dots+\frac1{q_n}+\frac1{q_\circ}\le1.
$$
Then there exists a unique multilinear continuous operator
$$
\tilde{\sf C}:B_{p_1,q_1}^{\alpha_1,0}\times\cdots\times B_{p_n,q_n}^{\alpha_n,0}\times B_{p_\circ,q_\circ}^{\alpha_\circ,0}
\to B_{p,q}^{\alpha,0}
$$
such that,
$$
\tilde{\sf C}(f_1,\dots,f_n,\xi)
={\sf C}(f_1,\dots,f_n,\xi)
$$
for any smooth inputs $(f_1,\dots,f_n,\xi)$.
\end{thm}

To show the main theorem, we introduce a regularity structure suitable for our context, and impose $B_{p,q}$ type bounds on models and modelled distributions.
Thus in our case, each basis vector $\tau$ of the model space has three homogeneity parameters $(\alpha_\tau,p_\tau,q_\tau)$, which is slightly different from the original setting \cite{Hai14}.
See \cite{HL17, LPT16, LPT18} for relevant studies.
In \cite{HL17, LPT16}, the authors defined $B_{p,q}$ type modelled distributions and proved a generalized reconstruction theorem, while $B_{\infty,\infty}$ type bounds are imposed on models.
In \cite{LPT18}, the authors defined $B_{p,p}$ (Sobolev) type models and modelled distributions to consider the Sobolev type rough paths.

This paper is organized as follows. In Section \ref{section 2}, we define some important notions used in this paper; Besov type norms, paraproducts, and the word Hopf algebra.
In Section \ref{section 3}, we show the Besov type estimates of Taylor remainders of iterated paraproducts.
In Section \ref{section 4}, we show the Besov type commutator estimates.


\section{Preliminaries}\label{section 2}

We introduce some important notions used throughout this paper.

\subsection{Besov type norms}

In this paper, we often use a sequence $\{a_j\}_{j=-1}^\infty$ of numbers, functions, or operators. We use simplifying notations for partial sums as follows.
$$
a_{<j}:=\sum_{i<j}a_i,\quad
a_{\ge j}:=\sum_{i\ge j}a_i.
$$

\begin{lmm}\label{3 lmm sum in ell^q}
Let $q\in[1,\infty]$ and let $\{c_j\}_{j=-1}^\infty$ be a sequence of nonnegative numbers. If $\alpha>0$, then we have
\begin{align}
\label{c_>}
\|\{2^{j\alpha}c_{\ge j}\}_{j}\|_{\ell^q}
&\lesssim\|\{2^{j\alpha}c_j\}_{j}\|_{\ell^q},\\
\label{c_<}
\|\{2^{-j\alpha}c_{\le j}\}_{j}\|_{\ell^q}
&\lesssim\|\{2^{-j\alpha}c_j\}_{j}\|_{\ell^q}.
\end{align}
\end{lmm}

\begin{proof}
We extend $\{c_j\}_{j=-1}^\infty$ into a sequence $\{c_j\}_{j\in\mathbb{Z}}$ by setting $c_j=0$ if $j\le-2$.
For \eqref{c_>}, by using Young's inequality on the group $\mathbb{Z}$,
\begin{align*}
\left\|\{2^{J\alpha} c_{\ge J}\}_{J}\right\|_{\ell^q}
&=\Big\|\Big\{\sum_{j\ge J}2^{(J-j)\alpha}2^{j\alpha}c_j\Big\}_{J}\Big\|_{\ell^q} \\
&\le \sum_{-\infty<i\le0}2^{i\alpha}\left\|\{2^{j\alpha} c_j\}_{j}\right\|_{\ell^q}
\lesssim \left\|\{2^{j\alpha} c_j\}_{j}\right\|_{\ell^q}.
\end{align*}
The proof of \eqref{c_>} is just an analogue.
\end{proof}

Denote by $\mcS=\mathcal{S}(\mathbb{R}^d)$ the space of Schwartz functions, and by $\mcS'$ its dual space.
Fix smooth radial functions $\chi$ and $\rho$ such that,
\begin{itemize}
\setlength{\itemsep}{1mm}
\item $\supp(\chi)\subset\{x\,;|x|<\frac43\}$ and $\supp(\rho)\subset\{x\,;\frac34<|x|<\frac83\}$,
\item $\chi(x)+\sum_{j=0}^\infty\rho(2^{-j}x)=1$ for any $x\in\mathbb{R}^d$.
\end{itemize}
Set $\rho_{-1}:=\chi$ and $\rho_j:=\rho(2^{-j}\cdot)$ for $j\ge0$. We define the Littlewood-Paley blocks
$$
\Delta_jf:=\mathcal{F}^{-1}(\rho_j\mathcal{F}f)
$$
for $f\in\mathcal{S}'$, where $\mathcal{F}$ is the Fourier transform and $\mathcal{F}^{-1}$ is its inverse.
It is useful to write
$$
\Delta_jf(x)=\int_{\bbR^d} Q_j(x,y)f(y)dy,
$$
where $Q_j(x,y)=\mathcal{F}^{-1}(\rho_j)(x-y)$. We also write $Q_j(h)=\mathcal{F}^{-1}(\rho_j)(h)$.

\begin{dfn}
For any $\alpha\in\mathbb{R}$ and $p,q\in[1,\infty]$, we define the (nonhomogeneous) Besov space $B_{p,q}^\alpha$ by the space of all $f\in\mcS'$ such that
$$
\|f\|_{B_{p,q}^\alpha}:=\left\| \left\{ 2^{j\alpha}\|\Delta_jf\|_{L^p} \right\}_{j\ge-1} \right\|_{\ell^q}<\infty.
$$
\end{dfn}

As stated in \cite[Theorem 2.36]{BCD11}, it is possible to define Besov norms without Littlewood-Paley theory. The aim of this paper is to study the following norm for two parameter functions.

\begin{dfn}
Let $\alpha\in\bbR$ and $p,q\in[1,\infty]$.
For any two parameter measurable function $\omega(x,y)$ on $\mathbb{R}^d\times\mathbb{R}^d$, define
$$
\|\omega\|_{D_{p,q}^\alpha}:=
\left\||h|^{-\alpha}\|\omega(x,x+h)\|_{L^p(dx)}\right\|_{L^q(dh/|h|^d)}.
$$
\end{dfn}

The following well-known result provides an alternative definition of Besov norm.
A self-contained proof appears in the next subsection.

\begin{prp}[{\cite[Theorem 2.36]{BCD11}}]\label{Besov equiv}
If $\alpha\in(0,1)$, then
$$
\|f\|_{B_{p,q}^\alpha}
\simeq \|f\|_{L^p}+\|\omega_f\|_{D_{p,q}^\alpha},
$$
where $\omega_f(x,y)=f(y)-f(x)$.
\end{prp}

\subsection{Technical lemmas}\label{tech sec}

We prove some technical lemmas used throughout this paper.

\begin{dfn}
Let $\alpha\in\bbR$ and $p,q\in[1,\infty]$.
For any sequence $\{f_j(x)\}_{j=-1}^\infty$ of measurable functions on $\bbR^d$, define
$$
\|\{f_j\}_j\|_{\mathbb{B}_{p,q}^\alpha}
:=\left\|\left\{ 2^{j\alpha}\|f_j\|_{L^p} \right\}_j\right\|_{\ell^q}.
$$
\end{dfn}

By definition, $\|f\|_{B_{p,q}^\alpha}=\|\{\Delta_jf\}_j\|_{\mathbb{B}_{p,q}^\alpha}$.
We often emphasize the variables $j$ and $x$ and write
$$
\|f_j(x)\|_{\mathbb{B}_{p,q}^\alpha}
=\left\|\left\{ 2^{j\alpha}\|f_j(x)\|_{L^p(dx)} \right\}_j\right\|_{\ell^q}.
$$
by an abuse of notation.

\begin{lmm}\label{nenai}
Let $\alpha>0$ and $q\in[1,\infty]$.
Let $F$ be a nonnegative function on $\bbR^d$ such that
$$
\big\||h|^{-\alpha}F(h)\big\|_{L^q(dh/|h|^d)}\le C.
$$
Then for any nonnegative function $\varphi\in\mcS$, one has
$$
\left\|2^{j\alpha}\int_{\bbR^d} 2^{jd}\varphi(2^jh)F(h)dh\right\|_{\ell^q}
\lesssim C.
$$
\end{lmm}

\begin{proof}
The proof is essentially contained in the latter half part of the proof of \cite[Theorem 2.36]{BCD11}.
There $\alpha\in(0,1)$ is assumed, but we can see that Lemma \ref{nenai} holds for any $\alpha>0$.
The only point to be modified is the integration over $2^j|h|>1$ when $q<\infty$. Indeed, for any $\varepsilon>0$,
\begin{align*}
2^{j\alpha}\left|\int_{2^j|h|>1} 2^{jd}\varphi(2^jh)F(h)dh\right|
&\le2^{-j\varepsilon}
\int_{2^j|h|>1}|2^jh|^{d+\alpha+\varepsilon}|\varphi(2^jh)|\frac{F(h)}{|h|^{\alpha+\varepsilon}}\frac{dh}{|h|^d}\\
&\lesssim 2^{-j\varepsilon}
\left(\int_{2^j|h|>1}\frac{F(h)^q}{|h|^{(\alpha+\varepsilon)q}}\frac{dh}{|h|^d}\right)^{1/q}
\end{align*}
by H\"older's inequality for the measure $dh/|h|^d$, and we have
\begin{align*}
\left\|2^{-j\varepsilon}\left(\int_{2^j|h|>1}\frac{F(h)^q}{|h|^{(\alpha+\varepsilon)q}}\frac{dh}{|h|^d}\right)^{1/q}\right\|_{\ell^q}\lesssim\big\||h|^{-\alpha}F(h)\big\|_{L^q(dh/|h|^d)},
\end{align*}
since the sum of $2^{-j\varepsilon q}$ over $j$ such that $2^j|h|>1$ is bounded by $|h|^{\varepsilon q}$.
\end{proof}

\begin{lmm}\label{nanikore}
Let $\alpha>0$.
For any $\omega\in D_{p,q}^\alpha$, one has the bound
$$
\big\| \Delta_{<j} \big(\omega(x,\cdot)\big) (x) \big\|_{\bbB_{p,q}^\alpha}
\lesssim\|\omega\|_{D_{p,q}^\alpha}.
$$
\end{lmm}

\begin{proof}
Since
\begin{align*}
\Delta_{<j} \big(\omega(x,\cdot)\big) (x)
&=\int_{\bbR^d} Q_{<j}(x,y)\omega(x,y)dy\\
&=\int_{\bbR^d} Q_{<j}(-h)\omega(x,x+h)dh,
\end{align*}
by using Minkowski's inequality and Lemma \ref{nenai} we have
\begin{align*}
\big\|\Delta_{<j} \big(\omega(x,\cdot)\big) (x) \big\|_{L^p}
&\lesssim \int_{\bbR^d} |Q_{<j}(-h)| \|\omega(x,x+h)\|_{L^p(dx)}dh\\
&\lesssim \|\omega\|_{D_{p,q}^\alpha}2^{-j\alpha}\mathbf{1}_j^q,
\end{align*}
where $\mathbf{1}_j^q$ denotes a sequence belonging to the unit sphere of $\ell^{q}$.
\end{proof}

Through this paper, we often use the notation $\mathbf{1}_j^q$ without notice.

\begin{lmm}\label{a lmm 1}
Let $\{\omega_j(x,y)\}_{j=-1}^\infty$ be a sequence of two parameter functions. Assume that for some $C>0$ and $\alpha>0$, the bound
\begin{align}\label{nazo}
\|\omega_j(x+h,x)\|_{\bbB_{p,q}^{\alpha-\theta}}
:=\big\|\big\{2^{j(\alpha-\theta)}\|\omega_j(x+h,x)\|_{L^p}\big\}_{j\ge-1}\big\|_{\ell^q} \le C|h|^\theta
\end{align}
holds for any $h\in\bbR^d$ and any $\theta$ in a neighborhood of $\alpha$.
Then $\omega=\sum_{j\ge-1}\omega_j$ converges in $D_{p,q}^\alpha$ and one has the bound
$$
\|\omega\|_{D_{p,q}^\alpha}\lesssim C.
$$
\end{lmm}

\begin{proof}
We follow the proof of \cite[Theorem 2.36]{BCD11}.
Since the case $q=\infty$ is the same as \cite[Lemma 3.7]{Hos19}, we consider $q<\infty$.
Assume $C\le1$ without loss of generality.
Let
\begin{align*}
A_N&=\left\{h\in\bbR^d\,;\,2^{-N-1}\le |h|<2^{-N}\right\}\quad (N\ge0),\\
A_{-1}&=\{h\in\bbR^d\,;\,1\le|h|\}.
\end{align*}
Fix a small $\varepsilon>0$ such that \eqref{nazo} holds for $\theta=\alpha\pm\varepsilon$. If $h\in A_N$ with $N\ge0$,
\begin{align*}
|h|^{-\alpha}\|\omega(x+h,x)\|_{L^p(dx)}
&\lesssim\sum_{j}|h|^{-\alpha}\|\omega_j(x+h,x)\|_{L^p(dx)}\\
&\lesssim \sum_{-1\le j<N} \mathbf{1}_j^q 2^{j\varepsilon}|h|^\varepsilon
+\sum_{j\ge N} \mathbf{1}_j^q2^{-j\epsilon}|h|^{-\varepsilon}.
\end{align*}
By H\"older's inequality with the weight $2^{\pm j\varepsilon}$, we have
\begin{align*}
\Big(\sum_{j<N} \mathbf{1}_j^q 2^{j\varepsilon} |h|^\varepsilon\Big)^q
&\lesssim |h|^{\varepsilon q}\Big(\sum_{j<N}2^{j\varepsilon}\Big)^{q-1}\sum_{j<N}\mathbf{1}_j^12^{j\varepsilon}\\
&\lesssim |h|^{\varepsilon q}2^{N\varepsilon(q-1)} \sum_{j<N}\mathbf{1}_j^12^{j\varepsilon},
\end{align*}
and
\begin{align*}
\Big(\sum_{j\ge N} \mathbf{1}_j^q 2^{-j\varepsilon} |h|^{-\varepsilon}\Big)^q
&\lesssim |h|^{-\varepsilon q}2^{-N\varepsilon(q-1)} \sum_{j\ge N}\mathbf{1}_j^12^{-j\varepsilon}.
\end{align*}
Since $|h|\sim2^{-N}$ on $A_N$, we have
\begin{align*}
&\int_{A_N}\Big(|h|^{-\alpha}\|\omega(x+h,x)\|_{L^p(dx)}\Big)^q\frac{dh}{|h|^d}\\
&\lesssim 2^{-N\varepsilon}\sum_{j<N}\mathbf{1}_j^12^{j\varepsilon}
+2^{N\varepsilon}\sum_{j\ge N}\mathbf{1}_j^12^{-j\varepsilon}
\lesssim \sum_{j\ge-1}2^{-|j-N|\varepsilon}\mathbf{1}_j^1.
\end{align*}
Summing them over $N\ge0$, by Young's inequality we have
$$
\sum_{N\ge0}\sum_{j\ge-1}2^{-|j-N|\varepsilon}\mathbf{1}_j^1<\infty.
$$
If $h\in A_{-1}$, similarly to above,
\begin{align*}
&\int_{A_{-1}}\Big(|h|^{-\alpha}\|\omega(x+h,x)\|_{L^p(dx)}\Big)^q\frac{dh}{|h|^d}\\
&\lesssim \sum_{j\ge-1}\mathbf{1}_j^12^{-j\varepsilon}\int_{|h|\ge1}|h|^{-\varepsilon q-d}dh
\lesssim 1,
\end{align*}
which completes the proof.
\end{proof}

Using above lemmas, we can prove Proposition \ref{Besov equiv}.

\begin{proof}[Proof of Proposition \ref{Besov equiv}]
Note that $\|\Delta_{-1}f\|_{L^p}\lesssim\|f\|_{L^p}$. For $j\ge0$, since $\int Q_j=0$ we have
$$
\Delta_jf(x)=\Delta_j(\omega_f(x,\cdot))(x).
$$
Lemma \ref{nanikore} yields $\|f\|_{B_{p,q}^\alpha}\lesssim \|f\|_{L^p}+\|\omega_f\|_{D_{p,q}^\alpha}$.
To show the converse, let $\omega_j(x,y)=\Delta_jf(y)-\Delta_jf(x)$ and apply Lemma \ref{a lmm 1}. Obviously,
\begin{align*}
\|\omega_j(x,x+h)\|_{L^{p}(dx)}
&\lesssim \|\Delta_jf\|_{L^{p}}
\lesssim 2^{-j\alpha} \mathbf{1}_j^{q}\|f\|_{B_{p,q}^\alpha}.
\end{align*}
By Minkowski's inequality and the continuity of the differentiation $B_{p,q}^{\alpha}\ni f\mapsto\nabla f\in(B_{p,q}^{\alpha-1})^{d}$ (see \cite[Proposition~2.78]{BCD11}), we have
\begin{align*}
\|\omega_j(x,x+h)\|_{L^{p}(dx)}&\le
\left\|h\cdot\int_0^1\nabla\Delta_jf(x+\theta h)d\theta\right\|_{L^{p}(dx)}\\
&\le|h|\|\Delta_j(\nabla f)\|_{L^{p}}\lesssim|h|2^{j(1-\alpha)}\mathbf{1}_j^{q}\|f\|_{B_{p,q}^\alpha}.
\end{align*}
By an interpolation, for any $\theta\in[0,1]$ we have
$$
\|\omega_j(x+h,x)\|_{L^{p}(dx)}
\lesssim|h|^\theta2^{j(\theta-\alpha)}\mathbf{1}_j^{q}\|f\|_{B_{p,q}^\alpha},
$$
so $\|\omega_f\|_{D_{p,q}^\alpha}\lesssim\|f\|_{B_{p,q}^\alpha}$ by Lemma \ref{a lmm 1}.\end{proof}

\subsection{Paraproduct}

For any smooth functions $f,g$, we can decompose the product $fg$ by
\begin{align*}
fg=\sum_{j,k\ge-1}\Delta_jf\Delta_kg
&=\Big(\sum_{j<k-1}+\sum_{|j-k|\le1}+\sum_{j+1<k}\Big)\Delta_jf\Delta_kg\\
&=:f\pl g+f\rs g+f\pr g.
\end{align*}
$f\pl g=g\pr f$ is called a \emph{paraproduct}, and $f\rs g$ is called a \emph{resonant}.
%

As in \cite{BH18}, we often use the two parameter extension of the paraproduct.
For any measurable function $\omega(x,y)$, we define
\begin{align*}
\mathsf{P}_j(\omega)(z)
&:=\iint_{\bbR^d\times\bbR^d} Q_{<j-1}(z,x)Q_j(z,y)\omega(x,y)dxdy
\end{align*}
and $\sfP(\omega):=\sum_j\sfP_j(\omega)$. Obviously, for the case $\omega(x,y)=f(x)g(y)$ we have $\sfP(\omega)=f\pl g$.

\begin{lmm}\label{a lmm 2}
Let $\alpha>0$.
If $\omega\in D_{p,q}^\alpha$, then $\sfP(\omega)\in B_{p,q}^\alpha$. The mapping $\omega\mapsto\sfP(\omega)$ is continuous.
\end{lmm}

\begin{proof}
In view of \cite[Proposition 8]{BH18},
$$
\|\sfP(\omega)\|_{B_{p,q}^\alpha}\lesssim
\|\sfP_j(\omega)\|_{\bbB_{p,q}^\alpha}.
$$
For the right hand side, exchanging variables $y=z+h$ and $x=z+h+k$ and using Minkowski's inequality, we have
\begin{align*}
\|\sfP_j(\omega)\|_{L^p}
&\lesssim\iint |Q_{<j-1}(-h)|\, |Q_j(-h-k) |\, \|\omega(z+h,z+h+k)\|_{L^p(dz)}dhdk\\
&=\int 2^{jd}\varphi(2^jk)\|\omega(z,z+k)\|_{L^p(dz)}dk,
\end{align*}
for some $\varphi\in\mcS(\bbR^d)$.
Thus Lemma \ref{nenai} completes the proof.
\end{proof}

\subsection{Word Hopf algebra}

We introduce a specific regularity structure.
Fix an integer $n$. For any integers $1\le k\le \ell\le n$, denote by $(k\dots \ell)$ the sequence from $k$ to $\ell$, which is called a \emph{word} throughout this paper.
We discuss the algebras made from the set $W$ of all such words.
Let $\spW$ be the commutative algebra freely generated by $W$ with unit $\mathbf{1}$.
We regard $\mathbf{1}$ as an empty word and consider the extended set $\words=W\cup\{\mathbf{1}\}$ of words.
For any nonempty words $\sigma=(k\dots\ell)$ and $\eta=((\ell+1)\dots m)$ in $W$, we define $\sigma\sqcup\eta=(k\dots m)$. We also define $\mathbf{1}\sqcup\tau=\tau\sqcup\mathbf{1}=\tau$.

\begin{dfn}
Define the linear map $\Delta:\spW\to\spW\otimes\spW$ by
\begin{align*}
\Delta\tau
&=\sum_{\sigma,\eta\in \words,\, \sigma\sqcup\eta=\tau}\sigma\otimes\eta
\end{align*}
for any $\tau\in \words$.
\end{dfn}

The map $\Delta$ is coassociative; $(\Delta\otimes\id)\Delta=(\id\otimes\Delta)\Delta$. It is easy to show the existence of the algebra map $A:\spW\to\spW$ such that
\begin{align*}
&A\mathbf{1}=\mathbf{1},\\
&M(A\otimes\id)\Delta\tau=M(\id\otimes A)\Delta\tau=0\quad(\tau\in W),
\end{align*}
where $M:\spW\otimes\spW\to\spW$ is the product map.
Such $A$ is called an antipode. In other words, $\spW$ is a \emph{Hopf algebra}.
The existence of $A$ yields that, the set $G$ of all algebra maps $\gamma:\spW\to\bbR$ forms a group by the product
$$
(\gamma_1*\gamma_2)(\tau)=(\gamma_1\otimes\gamma_2)\Delta\tau.
$$
The inverse of $\gamma\in G$ is given by $\gamma^{-1}=\gamma \circ A$.

In Section \ref{section 3}, we study the family $\{f_\tau=f_\tau(x)\}_{\tau\in W}$ of functions on $\bbR^d$, indexed by words.
We regard $f(x)\in G$ by extending the map $\tau\mapsto f_\tau(x)$ algebraically. Then we define the $G$-valued two parameter function by
$$
\omega(x,y)=f(x)^{-1}*f(y).
$$
In other words, we have a family $\{\omega_\tau(x,y):=\omega(x,y)(\tau)\}_{\tau\in W}$ of two parameter functions, indexed by words.
The following relationships between $f$ and $\omega$
are useful in Section \ref{section 3}.

\begin{lmm}
For any $1\le k\le \ell\le n$, one has
\begin{align}
\label{f to omega}
&\omega_{k\dots\ell}(x,y)
=f_{k\dots \ell}(y)-f_{k\dots\ell}(x) -\sum_{m=k}^{\ell-1}f_{k\dots m}(x)\omega_{(m+1)\dots \ell}(x,y),\\
\label{f to omega 2}
&
\begin{aligned}
\omega_{k\dots\ell}(x,z)
&=\omega_{k\dots\ell}(x,y)+\omega_{k\dots\ell}(y,z)\\
&\quad+\sum_{m=k}^{\ell-1} \omega_{k\dots m}(x,y)\omega_{(m+1)\dots\ell}(y,z).
\end{aligned}
\end{align}
\end{lmm}

\begin{proof}
Immediate consequences of the simple formulas.
\eqref{f to omega}: $f(y)=f(x)*\omega(x,y)$,
\eqref{f to omega 2}: $\omega(x,z)=\omega(x,y)*\omega(y,z)$.
\end{proof}


\section{Taylor remainders of iterated paraproducts}\label{section 3}

For a given sequence $f_1,f_2,\dots$ of functions, we define the \emph{iterated paraproducts}
$$
(f_1)^\pl:=f_1,\quad
(f_1,\dots,f_n)^\pl:=(f_1,\dots,f_{n-1})^\pl\pl f_n.
$$
The aim of this section is to show the following Besov type estimate, which is an extension of \cite[Theorem 3.1]{Hos19}. We write $f_{k\dots \ell}^\pl:=(f_k,\dots,f_\ell)^\pl$.

\begin{thm}\label{3 thm main}
For any measurable functions $f_1,\dots,f_n$, we define the family
$$
\{\omega_{k\dots \ell}^\pl(x,y)\}_{1\le k\le \ell\le n}
$$
of two parameter functions by the recursive formula \eqref{f to omega} with $f_{k\dots\ell}$ replaced by $f_{k\dots\ell}^\pl$.
Let $\alpha_1,\dots,\alpha_n\in(0,1)$, $p_1,\dots,p_n,q_1,\dots,q_n\in[1,\infty]$, and $f_i\in B_{p_i,q_i}^{\alpha_i}$ for each $i$.
If $\alpha:=\alpha_1+\cdots+\alpha_n<1$, $\frac1p=\frac1{p_1}+\cdots+\frac1{p_n}\le1$, and $\frac1q=\frac1{q_1}+\cdots+\frac1{q_n}\le1$, then we have $\omega_{1\dots n}^\pl\in D_{p,q}^\alpha$ and
\begin{align}\label{3 thm main bound}
\|\omega_{1\dots n}^\pl\|_{D_{p,q}^\alpha}
\lesssim\|f_1\|_{B_{p_1,q_1}^{\alpha_1}}\cdots\|f_n\|_{B_{p_n,q_n}^{\alpha_n}}.
\end{align}
\end{thm}

\subsection{Simplified iterated paraproducts}

Fix the parameters and the functions as in Theorem \ref{3 thm main}.
For any $1\le k\le\ell\le n$, we use the following simplifying notations.
\begin{align*}
&\alpha_{k\dots \ell}:=\alpha_k+\dots+\alpha_\ell,\quad
\frac1{p_{k\dots\ell}}:=\frac1{p_k}+\dots+\frac1{p_\ell},\quad
\frac1{q_{k\dots\ell}}:=\frac1{q_k}+\dots+\frac1{q_\ell}.
\end{align*}
First we show the existence of the family $\{\tif_{k\dots \ell}\}_{1\le k\le\ell\le n}$ such that the corresponding $\{\tiomega_{k\dots \ell}\}_{1\le k\le\ell\le n}$ satisfies the bound \eqref{3 thm main bound}.

\begin{dfn}
For any $j\ge-1$, we define
\begin{align*}
(\tif_k)_j:=\Delta_jf_k,\quad
(\tif_{k\dots\ell})_j:=(\tif_{k\dots(\ell-1)})_{<j-1}(\tif_\ell)_j,
\end{align*}
(the latter definition has a meaning only if $j\ge1$) and set
\begin{align*}
\tif_{k\dots \ell}=\sum_j(\tif_{k\dots \ell})_j.
\end{align*}
Moreover, we define the family $\{\tiomega_{k\dots\ell}\}_{1\le k\le\ell\le n}$ by the recursive formula \eqref{f to omega} with $f_{k\dots\ell}$ replaced by $\tif_{k\dots\ell}$.
\end{dfn}

We consider the decomposition $\tiomega_{k\dots\ell}=\sum_j(\tiomega_{k\dots\ell})_j$ as follows.
The proof of this lemma is left to the reader.

\begin{lmm}\label{3 lmm formulas}
Define $(\tiomega_{k\dots\ell})_j$ recursively by
\begin{align*}
(\tiomega_{k\dots\ell})_j(x,y)
=(\tif_{k\dots\ell})_j(y) - (\tif_{k\dots \ell})_j(x) 
 -\sum_{m=k}^{\ell-1}\tif_{k\dots m}(x)(\tiomega_{(m+1)\dots\ell})_j(x,y).
\end{align*}
Then one has the following formulas.
\begin{enumerate}
\item $(\tiomega_k)_j(x,y)=\Delta_jf_k(y)-\Delta_jf_k(x)$.
\item If $k<\ell$,
\begin{align*}
(\tiomega_{k\dots\ell})_j(x,y)
&=(\tiomega_{k\dots(\ell-1)})_{<j-1}(x,y)(\tif_\ell)_j(y) 
-(C_{k\dots(\ell-1)})_{\ge j-1}(x)(\tiomega_\ell)_j(x,y),
\end{align*}
where $(C_{k\dots\ell})_{1\le k\le\ell\le n}$ is recursively defined by $(C_k)_j(x)=\Delta_jf_k(x)$ and
$$
(C_{k\dots\ell})_j(x)=(\tif_{k\dots\ell})_j(x)-\sum_{m=k}^{\ell-1}\tif_{k\dots m}(x)(C_{(m+1)\dots\ell})_j(x).
$$
\item If $k<\ell$,
$$
(C_{k\dots\ell})_j(x)=-(C_{k\dots(\ell-1)})_{\ge j-1}(x)(\tif_\ell)_j(x).
$$
\end{enumerate}
\end{lmm}

\begin{proof}[Proof of the bound \eqref{3 thm main bound} for $\tiomega$]
Without loss of generality, we assume $\|f_i\|_{B_{p_i,q_i}^{\alpha_i}}\le1$ for any $i$.
To apply Lemma \ref{a lmm 1}, we show the bound
\begin{align}\label{3 eq key estimate of tiomega}
\big\| (\tiomega_{1\dots n})_j(x,x+h) \big\|_{\mathbb{B}_{p_{1\dots n},q_{1\dots n}}^{\alpha_{1\dots n}-\theta}}
\lesssim|h|^\theta
\end{align}
uniformly over $\theta$ in a neighborhood of $\alpha_{1\dots n}<1$.
The case $n=1$ is already proved in the proof of Proposition \ref{Besov equiv}, in Section \ref{tech sec}.
%
Let $n\ge2$.
By Lemma \ref{3 lmm formulas}-(3), we inductively have
\begin{align*}
&\big\| (C_{1\dots n})_j \big\|_{\bbB_{p_{1\dots n},q_{1\dots n}}^{\alpha_{1\dots n}}}\\
&\le \big\| (C_{1\dots (n-1)})_{\ge j-1} \big\|_{\bbB_{p_{1\dots (n-1)},q_{1\dots (n-1)}}^{\alpha_{1\dots (n-1)}}} \big\| (\tif_n)_j \big\|_{\bbB_{p_n,q_n}^{\alpha_n}}
\lesssim1,
\end{align*}
where we use Lemma \ref{3 lmm sum in ell^q}-\eqref{c_>} for the bound of $(C_{1\dots (n-1)})_{\ge j-1}$.
Assume \eqref{3 eq key estimate of tiomega} holds for the word $(1\dots(n-1))$, uniformly over $\theta\in(\alpha_{1\dots(n-2)},1]$. 
Then by Lemma \ref{3 lmm sum in ell^q}-\eqref{c_<}, we have
\begin{align*}
\big\| (\tiomega_{1\dots (n-1)})_{<j-1}(x,x+h) \big\|_{\mathbb{B}_{p_{1\dots (n-1)},q_{1\dots (n-1)}}^{\alpha_{1\dots (n-1)}-\theta}}
\lesssim|h|^\theta
\end{align*}
for any $\theta\in(\alpha_{1\dots(n-1)},1]$. For such $\theta$, by Lemma \ref{3 lmm formulas}-(2),
\begin{align*}
&\big\| (\tiomega_{1\dots n})_j(x,x+h) \big\|_{\mathbb{B}_{p_{1\dots n},q_{1\dots n}}^{\alpha_{1\dots n}-\theta}}\\
&\le\big\| (\tiomega_{1\dots (n-1)})_{<j-1}(x,x+h) \big\|_{\mathbb{B}_{p_{1\dots (n-1)},q_{1\dots (n-1)}}^{\alpha_{1\dots (n-1)}-\theta}}
\|(\tif_n)_j\|_{\mathbb{B}_{p_n,q_n}^{\alpha_n}}\\
&\quad +\big\|(C_{1\dots(n-1)})_{\ge j-1}\big\|_{\mathbb{B}_{p_{1\dots(n-1)},q_{1\dots(n-1)}}^{\alpha_{1\dots(n-1)}}} 
\big\|(\tiomega_n)_j\big\|_{\mathbb{B}_{p_n,q_n}^{\alpha_n-\theta}}
\lesssim|h|^\theta.
\end{align*}
Thus we have the required bound by an induction on $n$.
\end{proof}

\subsection{Proof of Theorem \ref{3 thm main}}

We show the bound \eqref{3 thm main bound} for $\omega^\pl$, which is really required.
For any word $\tau=(k\dots\ell)$, denote by $\Pi(\tau)$ the set of all \emph{partitions} of $\tau$, that is, we write
$$
\{\tau_1,\dots,\tau_m\}\in\Pi(\tau)
$$
if $\tau_1,\dots,\tau_m$ are nonempty words of the form $\tau_j=(k_j\dots\ell_j)$ for each $j$, where
$k_1=k$, $\ell_m=\ell$, and $\ell_j+1=k_{j+1}$ for any $j$.
Recall the definitions of $\alpha_\tau=\alpha_{k\dots\ell}$, $p_\tau$, and $q_\tau$ as before.

\begin{lmm}\label{3 lmm atomic decomposition}
For any word $\tau=(k\dots \ell)$, there exists a function $[\tif]_\tau \in B_{p_\tau,q_\tau}^{\alpha_\tau}$ continuously depending on $f_k,\dots,f_\ell$, such that, one has the formula
\begin{align}\label{3 eq atomic decomposition 1}
\tif_\tau
=\sum_{\{\sigma,\eta\}\in\Pi(\tau)} \tif_\sigma\pl[\tif]_\eta+[\tif]_\tau.
\end{align}
Moreover, one has the atomic decomposition
\begin{align}\label{3 eq atomic decomposition 2}
\tif_\tau=\sum_{m=1}^\infty\sum_{\{\tau_1,\dots,\tau_m\}\in\Pi(\tau)} ([\tif]_{\tau_1},\dots,[\tif]_{\tau_m})^{\pl}.
\end{align}
\end{lmm}

\begin{proof}
Second formula \eqref{3 eq atomic decomposition 2} is an immediate consequence of the first one \eqref{3 eq atomic decomposition 1}.
The proof of \eqref{3 eq atomic decomposition 1} is essentially the same as \cite[Proposition 12]{BH18}.
The point is that we use Besov norms $B_{p,q}^\alpha$, while in \cite{BH18} the particular case $p=q=\infty$ is considered.

Here we give a proof of \eqref{3 eq atomic decomposition 1}. Write $\omega_f(x,y)=f(y)-f(x)$ for simplicity.
Expanding $\tiomega_\tau$ by repeating \eqref{f to omega}, we have
\begin{align}
\begin{aligned}\label{confusing}
&\tiomega_\tau(x,y)\\
&=\omega_{\tif_\tau}(x,y)-\sum_{\{\tau_1,\tau_2\}\in\Pi(\tau)}\tif_{\tau_1}(x)\tiomega_{\tau_2}(x,y)\\
&=\cdots\\
&=\omega_{\tif_\tau}(x,y)
-\sum_{m=2}^\infty(-1)^m\sum_{\{\tau_1,\dots,\tau_m\}\in\Pi(\tau)}
(\tif_{\tau_1}\dots \tif_{\tau_{m-1}})(x)\omega_{\tif_{\tau_m}}(x,y).
\end{aligned}
\end{align}
Applying the two parameter operator $\mathsf{P}$ to both sides, we have
\begin{align*}
\mathsf{P}(\tiomega_\tau)
=1\pl \tif_\tau-\sum_{m=2}^\infty(-1)^m\sum_{\{\tau_1,\dots,\tau_m\}\in\Pi(\tau)}
(\tif_{\tau_1}\dots \tif_{\tau_{m-1}})\pl \tif_{\tau_m}.
\end{align*}
By Lemma \ref{a lmm 2}, $\sfP(\tiomega_\tau)$ belongs to $B_{p_\tau,q_\tau}^{\alpha_\tau}$ and continuously depends on $f_k,\dots,f_\ell$. 
If $\tif_{\tau_m}$ has a decomposition \eqref{3 eq atomic decomposition 1},
\begin{align*}
&\sum_{m=2}^\infty(-1)^m\sum_{\{\tau_1,\dots,\tau_m\}\in\Pi(\tau)}
(\tif_{\tau_1}\dots \tif_{\tau_{m-1}})\pl \tif_{\tau_m}\\
&=\sum_{m=2}^\infty(-1)^m\sum_{\{\tau_1,\dots,\tau_m\}\in\Pi(\tau)}
(\tif_{\tau_1}\dots \tif_{\tau_{m-1}})\pl [\tif]_{\tau_m}\\
&\quad+\sum_{m=2}^\infty(-1)^m\sum_{\{\tau_1,\dots,\tau_m,\tau_{m+1}\}\in\Pi(\tau)}
(\tif_{\tau_1}\dots \tif_{\tau_{m-1}})\pl (\tif_{\tau_m}\pl[\tif]_{\tau_{m+1}})\\
&=\sum_{\{\tau_1,\tau_2\}\in\Pi(\tau)}\tif_{\tau_1}\pl[\tif]_{\tau_2}\\
&\quad+\sum_{m=2}^\infty(-1)^m\sum_{\{\tau_1,\dots,\tau_m,\tau_{m+1}\}\in\Pi(\tau)}
\mathsf{R}
(\tif_{\tau_1}\dots \tif_{\tau_{m-1}}, \tif_{\tau_m}, [\tif]_{\tau_{m+1}}),
\end{align*}
where $\sfR$ is the correcting operator defined by
$$
\mathsf{R}(a,b,c):=a\pl (b\pl c) - (ab)\pl c.
$$
The sum of all $\sfR$ terms belongs to $B_{p_\tau,q_\tau}^{\alpha_\tau}$ and continuously depends on $f_k,\dots,f_\ell$. Its proof is left to Lemma \ref{a lmm 3} below.
Then we obtain the formula \eqref{3 eq atomic decomposition 1} since
$$
\|\tif_\tau - 1\pl \tif_\tau\|_{B_{p_\tau,q_\tau}^r}
=\|\Delta_{\le0}\tif_\tau\|_{B_{p_\tau,q_\tau}^r}
\lesssim\|\tif_\tau\|_{B_{p_\tau.q_\tau}^{\alpha_\ell}}
$$
for any $r>0$.
\end{proof}

\begin{lmm}\label{a lmm 3}
Let $\sigma=(k\dots\ell)$, $\alpha'>0$, $p',q'\in[1,\infty]$, and $g\in B_{p',q'}^{\alpha'}$. Assume that
$\alpha=\alpha_\sigma+\alpha'<1$, $1/p=1/{p_\sigma}+1/{p'}\le1$, and $1/q=1/{q_\sigma}+1/{q'}\le1$. Then one has the bound
\begin{align*}
&\left\|\sum_{m=2}^\infty(-1)^m\sum_{\{\tau_1,\dots,\tau_m\}\in\Pi(\sigma)}
\mathsf{R}
(\tif_{\tau_1}\dots \tif_{\tau_{m-1}}, \tif_{\tau_m}, g)\right\|_{B_{p,q}^\alpha}\\
&\lesssim \|f_k\|_{B_{p_k,q_k}^{\alpha_k}}\dots\|f_\ell\|_{B_{p_\ell,q_\ell}^{\alpha_\ell}}\|g\|_{B_{p',q'}^{\alpha'}}.
\end{align*}
\end{lmm}

\begin{proof}
Just an analogue of \cite[Proposition 10]{BH18}, so see it for details. In view of the formula \eqref{confusing}, it is sufficient to show that
\begin{align*}
\left\| \sfP_j \big((\tiomega_\sigma(x,\cdot)\pl g)(y) \big) \right\|_{\bbB_{p,q}^\alpha}<\infty,
\end{align*}
where we write $\sfP_j(\Omega)=\sfP_j(\Omega(x,y))$ as an abuse of notation. Since the integral $\int Q_j(z,y) Q_{<i-1}(y,u)Q_i(y,v)dy$ vanishes if $|i-j|\ge N$ for some constant $N$,
\begin{align*}
&\int Q_j(z,y) \big((\tiomega_\sigma(x,\cdot)\pl g)(y) dy\\
&=\sum_{i;|i-j|<N}\int Q_j(z,y) \Delta_{<i-1}(\tiomega_\sigma(x,\cdot))(y)\Delta_ig(y)dy.
\end{align*}
By the formula \eqref{f to omega 2},
\begin{align*}
\Delta_{<i-1}(\tiomega_\sigma(x,\cdot))(y)
&=\tiomega_\sigma(x,y)+\Delta_{<i-1}(\tiomega_\sigma(y,\cdot))(y)\\
&\quad+\sum_{\{\eta,\zeta\}\in\Pi(\sigma)}\tiomega_\eta(x,y)\Delta_{<i-1}(\tiomega_\zeta(y,\cdot))(y).
\end{align*}
Hence, by using Lemma \ref{nanikore}, we see that $\sfP_j \big((\tiomega_\sigma(x,\cdot)\pl g)(y) \big)(z)$ is a sum of the integrals of the form
\begin{align}\label{fundamental}
\sum_{i;|i-j|<N} \iint Q_{<j-1}(z,x)Q_j(z,y) A(x,y) B_i(y) C_i(y) dxdy,
\end{align}
where $A\in D_{p_A,q_A}^{\alpha_A}$, $B\in \bbB_{p_B,q_B}^{\alpha_B}$, $C\in \bbB_{p_C,q_C}^{\alpha_C}$, and parameters are such that $\alpha=\alpha_A+\alpha_B+\alpha_C$, $1/p=1/p_A+1/p_B+1/p_C$, and $1/q=1/q_A+1/q_B+1/q_C$. Exchanging variables $y=z+h$ and $x=z+h+k$, we see that the $L^p(dz)$ bound of such an integral is as follows.
\begin{align*}
&\sum_{i;|i-j|<N}\iint|Q_{<j-1}(-h-k)||Q_j(-h)|
\|A(z+k,z)\|_{L^{p_A}}\|B_i\|_{L^{p_B}}\|C_i\|_{L^{p_C}} dhdk\\
&\lesssim \int 2^{jd}K(2^jk) \|A(z+k,z)\|_{L^{p_A}} dk \ 2^{-j(\alpha_B+\alpha_C)}c_j,
\end{align*}
where $K\in \mcS(\bbR^d)$ and $\{c_j\}\in\ell^{q_{BC}}$ with $1/{q_{BC}}=1/q_B+1/q_C$.
By Lemma \ref{nenai}, we have that the above integral is bounded by $2^{-j\alpha}d_j$ with $\{d_j\}\in\ell^q$, which completes the proof.
\end{proof}

For any partition $\{\tau_1,\dots,\tau_m\}\in\Pi(\tau)$, we define
\begin{align}
\label{3 dfn omega_tau 1}
&[\tif]_{\tau_1\dots\tau_m}^\pl
=([\tif]_{\tau_1},\dots,[\tif]_{\tau_m})^\pl,\\
&
\begin{aligned}\label{3 dfn omega_tau 2}
[\tiomega]_{\tau_1\dots\tau_m}^\pl(x,y)
&=[\tif]_{\tau_1\dots \tau_m}^\pl(y)-[\tif]_{\tau_1\dots\tau_m}^\pl(x)\\
&\quad-\sum_{j=1}^{m-1}[\tif]_{\tau_1\dots\tau_j}^\pl(x)
[\tiomega]_{\tau_{j+1}\dots\tau_m}^\pl(x,y).
\end{aligned}
\end{align}
Summing \eqref{3 dfn omega_tau 2} over all $\{\tau_1,\dots,\tau_m\}\in\Pi(\tau)$, we can inductively obtain
\begin{align}\label{3 lmm partition}
\tiomega_\tau=\sum_{m=1}^\infty\sum_{\{\tau_1,\dots,\tau_m\}\in\Pi(\tau)}
[\tiomega]^\pl_{\tau_1\dots\tau_m}
=:\sum_{\Xi\in\Pi(\tau)}[\tiomega]_\Xi^\pl.
\end{align}

\begin{proof}[Proof of Theorem \ref{3 thm main}]
We emphasize the dependence of $\omega_{k\dots \ell}^\pl$ on $f_k\dots,f_\ell$ by writing
$$
\omega_{k\dots\ell}^\pl=\omega^\pl(f_k,\dots,k_\ell).
$$
We prove the result by an induction on the number of the components of $\omega^\pl$.
By the formula \eqref{3 eq atomic decomposition 1}, $[\tif]_{(k)}=f_k$ for a word with only one letter.
Hence if $\Xi\in\Pi(\tau)$ has the same cardinality as the length of $\tau$ (denoted by $|\Xi|=|\tau|$), we have $[\tif]_\Xi^\pl=f_\tau^\pl$ and $[\tiomega]_\Xi^\pl=\omega_\tau^\pl$.
Hence by \eqref{3 lmm partition},
$$
\omega_\tau^\pl
=\tiomega_\tau-\sum_{\Xi\in\Pi(\tau), |\Xi|<|\tau|}[\tiomega]_\Xi^\pl.
$$
The bound for $\tiomega_\tau$ was already obtained.
By an assumption of the induction, $\omega^\pl$ is continuous as a less than $|\tau|$-component operator. Hence
\begin{align*}
\big\|[\tiomega]_{\tau_1\dots\tau_m}^\pl\big\|_{D_{p_\tau,q_\tau}^{\alpha_\tau}}
&=\big\|\omega^\pl([\tif]_{\tau_1},\dots,[\tif]_{\tau_m})\big\|_{D_{p_\tau,q_\tau}^{\alpha_\tau}}\\
&\lesssim\|[\tif]_{\tau_1}\|_{B_{p_{\tau_1},q_{\tau_1}}^{\alpha_{\tau_1}}}
\dots\|[\tif]_{\tau_m}\|_{B_{p_{\tau_m},q_{\tau_m}}^{\alpha_{\tau_m}}}\\
&\lesssim\|f_k\|_{B_{p_k,q_k}^{\alpha_k}}\dots\|f_\ell\|_{B_{p_\ell,q_\ell}^{\alpha_\ell}},
\end{align*}
where we use the continuity of $(f_{k_j},\dots,f_{\ell_j})\mapsto[\tif]_{\tau_j}$ (Lemma \ref{3 lmm atomic decomposition}).
As a result, we obtain the continuity of $\omega_\tau^\pl$ with respect to $f_k,\dots,f_\ell$.
\end{proof}


\section{Besov type regularity structure and commutator estimates}\label{section 4}

We prove Theorem \ref{4 thm commutator estimate} in the rest of this paper.
We show only the existence of the continuous map $\tilde{\sf C}$. The uniqueness of $\tilde{\sf C}$ and its multilinearity follows from the denseness argument.

\subsection{Besov type regularity structure}

We return to the Hopf algebra $\spW$ with the character group $G$.
We consider a subset
$$
V=\{\mathbf{1}\} \cup \{ (k\dots n) \ ;\,  k=1,\dots,n\}.
$$
of $\words$ and a linear subspace $T=\langle V\rangle$.
Since $\Delta T\subset\spW\otimes T$, for any $\gamma\in G$ we can define the linear map $\Gamma_\gamma:T\to T$ by
$$
\Gamma_\gamma=(\gamma\otimes\id)\Delta.
$$
The pair $(T,G)$ is an example of the \emph{regularity structure}.

\begin{rem}
Note that the position of $\gamma$ is opposite to the original definition \cite{Hai14}.
Because of it, in the mapping $\gamma\mapsto\Gamma_\gamma$, the order of multiplication is turned over as follows.
\begin{align}\label{tired}
\Gamma_{\gamma_1}\Gamma_{\gamma_2}=\Gamma_{\gamma_2*\gamma_1}
\end{align}
\end{rem}

We define a model $(\Pi,\Gamma)$ on the regularity structure $(T,G)$.
Fix the parameters and the functions satisfying the assumptions in Theorem \ref{4 thm commutator estimate}.
For any $1\le k\le n$, we define
\begin{align*}
&\alpha_{k\dots n}'=\alpha_k+\dots+\alpha_n+\alpha_\circ,\quad
\frac1{p_{k\dots n}'}=\frac1{p_k}+\dots+\frac1{p_n}+\frac1{p_\circ},\quad
\frac1{q_{k\dots n}'}=\frac1{q_k}+\dots+\frac1{q_n}+\frac1{q_\circ}.
\end{align*}
Moreover, set $\alpha_\mathbf{1}'=\alpha_\circ$, $p_\mathbf{1}'=p_\circ$, and $q_\mathbf{1}'=q_\circ$.

\begin{dfn}
Let $f_{k\dots \ell}^\pl=(f_k,\dots,f_\ell)^\pl$ be the iterated paraproduct.
We regard $f^\pl(x)\in G$ by extending the map $\tau\mapsto f_\tau^\pl(x)$ algebraically, and define
$$
\omega^\pl(x,y)=f^\pl(x)^{-1}*f^\pl(y),\quad
\Gamma_{xy}=\Gamma_{\omega^\pl(x,y)}.
$$
\end{dfn}


\begin{dfn}
For any linear map $\Pi:T\to\mcS'$, define
\begin{align*}
\Pi_x\tau=(f^\pl(x)^{-1}\otimes\Pi)\Delta\tau.
\end{align*}
Denote by $\mcM$ the set of all maps $\Pi$ such that
\begin{align*}
\|\Pi\|_{\mcM}
:=\sup_{\tau\in V}\big\| \Delta_{<j}( \Pi_x\tau ) (x) \big\|_{\bbB_{p_\tau',q_\tau'}^{\alpha_\tau'}}<\infty.
\end{align*}
\end{dfn}

It is easy to show the following formulas.
\begin{align}\label{sleepy}
\Gamma_{yx}\Gamma_{zy}=\Gamma_{zx},\quad
\Pi_y\Gamma_{xy}=\Pi_x.
\end{align}
The pair $(\Pi,\Gamma)$ is called a \emph{model} on the regularity structure $(T,G)$.
Note that these formulas are slightly different from the original ones \cite{Hai14}, like the formula \eqref{tired}.

As an analogue of \cite{Hos19,BH19}, we can show that the space $\mcM$ has a simple topological structure.
Let 
$$
V^-=\{\mathbf{1}\}\cup\{(k\dots n)\, ;\, k=2,\dots,n\}.
$$
Note that $\alpha_\tau'<0$ for any $\tau\in V^-$ by assumption.

\begin{thm}
For any $\Pi\in\mcM$, define the linear map $[\Pi]:T\to\mcS'$ by
\begin{align}\label{Pi to [Pi]}
\Pi\tau=\sum_{\sigma\sqcup\eta=\tau,\, \sigma\neq\mathbf{1}}f_\sigma^\pl\pl[\Pi]\eta+[\Pi]\tau.
\end{align}
Then $[\Pi]\tau\in B_{p_\tau',q_\tau'}^{\alpha_\tau'}$ and the mapping
$$
(f_1,\dots,f_n,\Pi)\mapsto[\Pi]\tau\in B_{p_\tau',q_\tau'}^{\alpha_\tau'}
$$
is continuous.
Conversely, for any given family
$$
\{[\Pi]\tau\}_{\tau\in V^-} \in \prod_{\tau\in V^-}B_{p_\tau',q_\tau'}^{\alpha_\tau'},
$$
there exists a unique element $\Pi\in\mcM$ satisfying \eqref{Pi to [Pi]}.
Moreover, the map $\{[\Pi]\tau\}_{\tau\in V^-}\mapsto\Pi$ is continuous.
\end{thm}

The above theorem is just an analogue of \cite[Theorem 14 and Corollary 15]{BH18}, so we leave the details to the reader. The only modification is that we have to use the Besov type reconstruction theorem in Appendix.

\subsection{Proof of Theorem \ref{4 thm commutator estimate}}

Now we show the iterated commutator estimates.
This part is strictly an analogue of \cite[Section 4]{Hos19}.

\begin{proof}[Proof of Theorem \ref{4 thm commutator estimate}]
For any given $\xi\in B_{p_\circ,q_\circ}^{\alpha_\circ}$, we can define $\Pi^\xi\in\mcM$ by
$$
[\Pi^\xi]\mathbf{1}=\xi,\quad
[\Pi^\xi](k\dots n)=0\quad (2\le k\le n).
$$
Note that
\begin{align}
\begin{aligned}\label{para from of xi}
&\Pi^\xi\mathbf{1}=\xi,\quad
\Pi^\xi(k\dots n)=f_{k\dots n}^\pl\pl\xi\quad (2\le k\le n),\\
&\Pi^\xi(1\dots n)=f_{1\dots n}^\pl\pl\xi+[\Pi^\xi](1\dots n),
\end{aligned}
\end{align}
by the formula \eqref{Pi to [Pi]}. 
Then the map
\begin{align*}
(f_1,\dots,f_n,\xi)\mapsto [\Pi^\xi](1\dots n)
\end{align*}
is continuous, which turns out to be the required map $\tilde{\sf C}$. It remains to show that
\begin{align}\label{netai}
[\Pi^\xi](1\dots n) = \sfC(f_1,\dots,f_n,\xi)
\end{align}
if all inputs $(f_1,\dots,f_n,\xi)$ are in $\mcS(\bbR^d)$.
Since $\Pi^\xi=(f^\pl(x)\otimes\Pi_x^\xi)\Delta$,
\begin{align}
\begin{aligned}\label{nemui}
\Pi^\xi(k\dots n)(x)
&=\Pi_x^\xi(k\dots n)(x)+f_{k\dots n}^\pl(x)\xi(x)\\
&\quad+\sum_{\ell=k}^{n-1} f_{k\dots\ell}^\pl(x)\left(\Pi_x^\xi((\ell+1)\dots n) \right)(x),
\end{aligned}
\end{align}
for any $1\le k\le n$.
By using it and \eqref{para from of xi}, we can inductively show that
$$
\Pi_x^\xi(k\dots n)(x)=-\sfC(f_k,\dots,f_n,\xi)(x)
$$
for $2\le k\le n$.
Then letting $k=1$ in \eqref{nemui} and using
$$
\Pi_x^\xi(1\dots n)(x)=\lim_{j\to\infty}\Delta_{<j}(\Pi_x^\xi(1\dots n))(x)=0
$$
because $\Delta_{<j}(\Pi_x^\xi(1\dots n))(x)\in\bbB_{p_{1\dots n}',q_{1\dots n}'}^{\alpha_{1\dots n}'}$ and $\alpha_{1\dots n}'>0$, we have \eqref{netai} by the definition of $\sfC$.
\end{proof}


\appendix

\section{Besov type reconstruction theorem}

We define Besov type modelled distributions. Recall that $V=\{\mathbf{1}\} \cup \{ (k\dots n) \ ;\,  k=1,\dots,n\}.$ and $T=\langle V\rangle$.

\begin{dfn}
For any function $\bsg:\bbR^d\to T$, define
$$
\omega^\bsg(x,y)=\bsg(y)-\Gamma_{xy}\bsg(x)
$$
and denote by $\omega_\tau^\bsg(y,x)$ its $\tau$-component.
Let $k$ be the smallest integer such that $\omega_{k\dots n}^\bsg(y,x)$ does not vanish,
and let $\alpha>\alpha_{k\dots n}'$, $p\in[1,p_{k\dots n}']$, and $q\in[1,q_{k\dots n}']$. For such parameters, we define
\begin{align*}
\|\bsg\|_{\mcD_{p,q}^\alpha}
:=\sup_\tau\big\| \omega_\tau^\bsg \big\|_{D_{p\setminus p_\tau', q\setminus q_\tau'}^{\alpha-\alpha_\tau'}},
\end{align*}
where $\frac1{p\setminus p_\tau'}=\frac1p-\frac1{p_\tau'}$, $\frac1{q\setminus q_\tau'}=\frac1q-\frac1{q_\tau'}$.
Let $\mcD_{p,q}^\alpha$ be the set of functions $\bsg:\bbR^d\to T$ such that $\|\bsg\|_{\mcD_{p,q}^\alpha}<\infty$.
\end{dfn}

Such $\bsg$ is called a \emph{modelled distribution} controlled by $\Gamma$.
We show the Besov type reconstruction theorem.

\begin{prp}\label{4 prp Besov type reconst}
For any $\bsg\in\mcD_{p,q}^\alpha$ and $\Pi\in\mcM$, we define
$$
\mcP\bsg(z)=\sum_j\iint_{\bbR^d\times\bbR^d} P_j(z,x)Q_j(z,y)\Pi_x\big(\bsg(x)\big)(y) dxdy.
$$
\begin{enumerate}
\item If $\alpha>0$, there exists a unique continuous bilinear map $\mcQ:\mcD_{p,q}^\alpha\times\mcM\to B_{p,q}^\alpha$ such that
$$
\big\| \Delta_{<j}(\mcP\bsg+\mcQ\bsg - (\Pi_x\bsg(x)) (x) ) \big\|_{\bbB_{p,q}^\alpha} <\infty.
$$
\item If $\alpha<0$,
$$
\big\| \Delta_{<j}(\mcP\bsg - (\Pi_x\bsg(x)) (x) ) \big\|_{\bbB_{p,q}^\alpha} <\infty.
$$
\end{enumerate}
(The operator $\mcR$ defined by $\mcR\bsg=\mcP\bsg+\mcQ\bsg$ if $\alpha>0$ and $\mcR\bsg=\mcP\bsg$ if $\alpha<0$
is called a \emph{reconstruction operator}.)
\end{prp}

\begin{proof}
The proof is almost the same as \cite[Proposition 9]{BH18}. In view of it, here it is sufficient to show the bound
\begin{align*}
\big\| \Delta_j\big(\mcP\bsg - \Pi_x\bsg(x) \big)(x) \big\|_{\bbB_{p,q}^\alpha}
<\infty.
\end{align*}
We have only to consider $j\ge1$. For such $j$,
\begin{align*}
&\Delta_j\big( \mcP\bsg - \Pi_x\bsg(x) \big)(x)\\
&=\sum_{i;i\sim j}\iiint Q_j(x,y)Q_{<i-1}(y,u)Q_i(y,v)\big( \Pi_u\bsg(u) - \Pi_x\bsg(x)\big)(v)dydudv,
\end{align*}
where $i\sim j$ means that $|i-j|\le N$ for some constant $N$. By \eqref{sleepy},
\begin{align*}
\Pi_u\bsg(u) - \Pi_x\bsg(x)
&=\Pi_x(\Gamma_{ux}\bsg(u)-\bsg(x))
=-\sum_{\tau\in V}\omega_\tau^\bsg(u,x)\Pi_x\tau.
\end{align*}
Hence the above integral is equal to
\begin{align*}
-\sum_{i;i\sim j}\sum_\tau \int Q_j(x,y)
\Delta_{<i-1}(\omega_\tau^\bsg(\cdot,x))(y)
\Delta_i(\Pi_x\tau)(y) dy.
\end{align*}
For the $\omega^\bsg$ part, since
\begin{align*}
\omega^\bsg(u,x)
&=\bsg(x)-\Gamma_{ux}\bsg(u)
=\bsg(x)-\Gamma_{yx}\bsg(y) +\Gamma_{yx}(\bsg(y) -\Gamma_{uy}\bsg(u))\\
&=\omega^\bsg(y,x) + \Gamma_{yx}\omega^\bsg(u,y),
\end{align*}
we have
\begin{align*}
\Delta_{<i-1}(\omega_\tau^\bsg(\cdot,x))(y)
=\omega_\tau^\bsg(y,x)+
\sum_\sigma \omega_\sigma^\pl(y,x)   \Delta_{<i-1}(\omega_{\sigma\sqcup\tau}^\bsg(\cdot,y))(y).
\end{align*}
Similarly, for the $\Pi$ part,
\begin{align*}
\Delta_i(\Pi_x\tau)(y)
=\Delta_i(\Pi_y\Gamma_{xy}\tau)(y)
=\sum_{\sigma\sqcup\eta=\tau}\omega_\sigma^\pl(x,y)\Delta_i(\Pi_y\eta)(y).
\end{align*}
Hence it turns out that $\Delta_j\big(\mcP\bsg - \Pi_x\bsg(x) \big)(x)$ is a sum of the integrals of the form
\begin{align*}
\sum_{i;i\sim j} \int Q_j(x,y) A(x,y) B_i(y) C_i(y) dy,
\end{align*}
where $A\in D_{p_A,q_A}^{\alpha_A}$, $B\in \bbB_{p_B,q_B}^{\alpha_B}$, $C\in \bbB_{p_C,q_C}^{\alpha_C}$, and parameters are such that $\alpha=\alpha_A+\alpha_B+\alpha_C$, $1/p=1/p_A+1/p_B+1/p_C$, and $1/q=1/q_A+1/q_B+1/q_C$.
Since we are in exactly the same situation as the previous one \eqref{fundamental}, we can complete the proof by a similar way to Lemma \ref{a lmm 3}.
\end{proof}

\vspace{7mm}
\noindent
{\bf Acknowledgements.}
The author is supported by JSPS KAKENHI Early-Career Scientists 19K14556.
The author thanks the anonymous referee for reading the paper carefully and providing helpful comments.


\end{document}